\documentclass[11pt]{article}

\usepackage{tikz}
\usepackage{subfigure}
\usepackage[english]{babel}

\usepackage[center]{caption2}
\usepackage{amsfonts,amssymb,amsmath,latexsym,amsthm}
\usepackage{multirow}

\newtheorem{thm}{Theorem}

\newtheorem{lem}[thm]{Lemma}
\newtheorem{prop}{Proposition}
\newtheorem{conj}[thm]{Conjecture}

\newtheorem{claim}{Claim}

\begin{document}

\title{An Erd\H{o}s-Gallai-type theorem for keyrings with larger number of leaves
\thanks{The work was supported by NNSF of China (No. 11671376), NSF of Anhui Province (No. 1708085MA18), and Anhui Initiative in Quantum Information Technologies (AHY150200).}
}
\author{Xinmin Hou$^a$,\quad Xiaodong Xue$^b$\\
\small $^{a,b}$ Key Laboratory of Wu Wen-Tsun Mathematics\\
\small School of Mathematical Sciences\\
\small University of Science and Technology of China\\
\small Hefei, Anhui 230026, China.
}

\date{}

\maketitle

\begin{abstract}
A keyring is a graph obtained from a cycle  by appending $r\ge0$ leaves to one of its vertices. Sidorenko proved an Erd\H{o}s-Gallai-type theorem: Every graph of order $n$ and size more than $\frac{(k-1)n}{2}$ contains a keyring of size at least $k$ and with $r$ leaves for  $r\le\frac{k-1}2$ (Theorem 1.4, An Erd\H{o}s-Gallai-type theorem for keyrings, Graphs Combin., 2018). In this note, we show that Sidorenko's theorem holds for larger $r$ and so complete the Erd\H{o}s-Gallai-type theorem for keyrings.
\end{abstract}

\section{Introduction}
The graphs considered in this paper are simple and undirected. Let $G=(V,E)$ be a graph.
We call $|V(G)|$ and $e(G)=|E(G)|$ the order and the size of $G$, respectively. For two disjoint sets of vertices $X$ and $Y$, write $E_{G}(X,Y)$ for the set of edges in $G$ with one endpoint in $X$ and the other in $Y$ and denote $e_{G}(X,Y)=|E_{G}(X,Y)|$.
For $u\in{V(G)}$, denote by $N_{G}(u)$ the set of neighbors of $u$ in $G$, and $d_{G}(u)=|N_{G}(u)|$ is the degree of $u$ in $G$. For a subset of vertices $X\subset{V(G)}$, write $G[X]$ for the subgraph of $G$ induced by $X$.

Erd\H{o}s and Gallai~\cite{EG59} proved the following statements.
\begin{thm}[Erd\H{o}s-Gallai Theorem, 1959]\label{THM: EG}
(1) Every graph of order $n$ and size more than $\frac{(k-1)n}{2}$ contains a path of size $k$.

(2) Every graph of order $n$ and size more than $\frac{(k-1)n}{2}$ contains a cycle of size at least $k$.
\end{thm}
Let $\mathcal{D}_{n,k}$ be the class of all graphs of order $n$ and size more than $\frac{(k-1)n}{2}$. Erd\H{o}s-Gallai Theorem told us every member of $\mathcal{D}_{n,k}$ contains a path of size $k$ and a cycle of size at least $k$.
Paths and cycles are not the only subgraphs contained in members of $\mathcal{D}_{n,k}$. The most famous open problem on this subject was  proposed by Erd\H{o}s and S\'os~\cite{ES65}.
\begin{conj}[Erd\H{o}s-S\'os Conjecture, 1965]
Every graph in $\mathcal{D}_{n,k}$ contains any tree of size $k$.
\end{conj}
There are fruitful partial results in the study of the conjecture, it has been proved that the conjecture holds  for many special families of trees on $k$ vertices, such as
\begin{itemize}
\item[(a)]  paths (Erd\H{o}s and Gallai~\cite{EG59}),
\item[(b)] spiders (Fan~\cite{Fan13}, Fan and Sun~\cite{Fan-Sun07}, Fan, Hong and Liu~\cite{AFC}),
\item[(c)] trees of diameter at most four (McLennan~\cite{Mc05}), and
\item[(d)] trees with a vertex joined to at least $\lfloor\frac{k}{2}\rfloor-1$ vertices of degree one (Sidorenko~\cite{AFA5}).
\end{itemize}
As pointed out in~\cite{K2k}, Tur\'an asked the very problems for which the Erd\H{o}s-Gallai theorem provides the answers.
Along this direction, we consider the  Erd\H{o}s-Gallai theorem of type (2). Except those special families of trees, as we have known, every member of $\mathcal{D}_{n,k}$ contains a subgraph listed in the following.
\begin{itemize}
\item[(i)]
A cycle of size at least $k$   (Erd\H{o}s and Gallai~\cite{EG59}).

\item[(ii)]
A lasso of size at least $k$ (Fan and Sun~\cite{Fan-Sun07}). {\em A lasso } is a graph obtained from a cycle by appending a path to one of its vertices.

\item[(iii)]
A keyring of size at least $k$ and with $r$ leaves for $r\le \frac{k-1}2$ (Sidorenko, Theorem 1.4 in~\cite{K2k}).  A vertex of a graph is {\em a leaf} if it has degree
one. {\em A keyring} is a graph obtained from a cycle by appending $r$ leaves to one of its vertices.
\end{itemize}

Note that a keyring of size $k$ can have leaves up to $k-3$. In this note,  we  show that Sidorenko's result still holds for larger $r$.
\begin{thm}\label{THM: Main}
Let $\frac{k}{2}\leq{r}\leq{k-3}$ and $G$ be a graph on $n$ vertices. If $e(G)>\frac{(k-1)n}{2}$ then $G$ contains a keyring with $r$ leaves and at least $k$ edges.
\end{thm}

Combining with the Erd\H{o}s-Gallai thoerem, (iii) given by Sidorenko, we have the following Erd\H{o}s-Gallai-type theorem.
\begin{thm}\label{THM: Keyring}
Every graph of order $n$ and size more than $\frac{(k-1)n}{2}$ contains a keyring  of size at least $k$ and with $r$ leaves for $0\le r\le k-3$.
\end{thm}

In Section 2, we give a structural lemma for every member of $\mathcal{D}_{n,k}$ and the proof of Theorem~\ref{THM: Main}. Some remarks will be given in Section 3. 

\section{Proof of Theorem~\ref{THM: Main}}
We define a class  $\mathcal{S}(G)$ of sets of vertices. A subset $X$ of $V(G)$ is in $\mathcal{S}(G)$ if and only if $X$ satisfies the following  properties:

\noindent{\bf (a)}: For any vertex $u\in X$, $G[X]$ contains a cycle containing $u$ of length at least $\frac{k}{2}$.

\noindent{\bf (b)}: For any two distinct vertices $u, v\in X$, $G[X]$ contains a path of length at least $\frac{k}{2}$ connecting $u$ and $v$.

For $\mathcal{S}(G)$, we have the following proposition.
\begin{prop}\label{F1}
(i) If $X, Y\in\mathcal{S}(G)$ and $X\cap{Y}\not=\emptyset$ then $X\cup Y\in\mathcal{S}(G)$.

(ii) If $X\in\mathcal{S}(G)$ and $P$ is a path in $G$ with both end vertices in $X$ then $X\cup V(P)\in\mathcal{S}(G)$.

(iii) If $C$ is a cycle of length at least $k$ then $V(C)\in\mathcal{S}(G)$.
\end{prop}
\begin{proof}
(i) The property (a) is apparently true. Now let $u,v\in X\cup Y$. If $u$ and $v$ are both in $X$ or both in $Y$, then $G[X]$ or $G[Y]$ contains a path  of length at least $\frac{k}{2}$ connecting $u$ and $v$ and so does $G[X\cup{Y}]$. So, without loss of generality, assume that $u\in{X\setminus{(X\cap Y)}}$ and $v\in Y\setminus(X\cap Y)$. Choose $P_{u}$ be a path in $G[X]$ connecting $u$ and a vertex $w\in X\cap Y$ such that $V(P_{u})\cap Y=\{w\}$, this can be done since $G[X]$ is connected. Next, let $P_{v}$ be a path connecting $w$ and $v$ in $G[Y]$ of length at least $\frac k2$. Thus $V(P_{u})\cap V(P_{v})=\{w\}$ and so the path $P=P_{u}\cup P_{v}$ is a path of length at least $\frac{k}{2}$ connecting $u$ and $v$ in $G[X\cup Y]$.

 (ii) Suppose $P=x_{0}x_{1}\ldots x_{k}$ is a path of length $k$. We say $x_{i}$ is on the left of $x_{j}$ if $i<j$, and on the right otherwise. For the set of vertices $X$ and path $P=x_{0}x_{1}\ldots x_{k}$ with $x_0, x_{k}\in X$ and an internal vertex $u=x_m$ of $P$, denote by $i(u)$ the largest $i<m$ with $x_{i}\in X$, and by $j(u)$ the smallest $j>m$ with $x_{j}\in X$. Apparently $i(u)$ and $j(u)$ exist and $i(u)<j(u)$. Denote $\ell(u)=x_{i(u)}$ and $r(u)=x_{j(u)}$. Then the two segments $P_{\ell}$ from $\ell(u)$ to $u$ and $P_{r}$ from $u$ to $r(u)$ on $P$ are internal vertex disjoint and have no internal vertex in $X$.

Choose $u\in X\cup V(P)$. If $u\in{X}$ then there is a cycle of length at least $\frac{k}{2}$ in $G[X]$ containing $u$. We are done. So suppose that $u\in V(P)$ but $u\not\in{X}$. Then $u$ is an internal vertex of $P$ and so $\ell(u)$ and $r(u)$ are two different vertices in $X$. So we can find a path $P'$ connecting $\ell(u)$ and $r(u)$ in $G[X]$ of length at least $\frac{k}{2}$. Combining with the two segments $P_{\ell}$ and $P_{r}$ on $P$, we get a desired cycle in $G[X\cup V(P)]$ through $u$.

Choose $u,v\in X\cup V(P)$. If both $u$ and $v$ are in $X$ then $G[X]$ contains a path of length at least $\frac{k}{2}$ connecting $u$ and $v$ and we are done. If,   without loss of generality,  $v\in X$ and $u$ is an internal vertex of $P$, then $\ell(u), r(u)\in X$ and at least one of them is different to $v$, say, $\ell(u)\not=v$. Thus there is a path $P_{v}$  of length at least $\frac{k}{2}$ connecting $v$ and $\ell(u)$ in $G[X]$. Combining with the segment $P_{\ell}$ on $P$ connecting $\ell(u)$ and $u$, we get a path of length at least $\frac{k}{2}$ connecting $u$ and $v$ in $G[X\cup V(P)]$.
Now assume neither $u$ nor $v$ are in $X$. Then they are both internal vertices of $P$.  Assume that $u$ is on the left of $v$. Then $\ell(u), r(v)\in X$ and $\ell(u)$ is on the left of $r(v)$ too. So $G[X]$ contains a path $P'$ connecting $\ell(u)$ and $r(v)$ of length at least $\frac{k}{2}$. Together with the two segments $P_u$ connecting $\ell(u)$ and $u$ and $P_v$ connecting $r(v)$ and $v$, we get a desired path in $G[X\cup V(P)]$ connecting $u$ and $v$.

(iii) It can be checked directly.
\end{proof}

Let $C$ be a cycle in $G$ of length at least $k$. Then $V(C)\in\mathcal{S}(G)$ by (iii) of Proposition~\ref{F1}.
Define $H(C)$ be the maximal set in $\mathcal{S}(G)$ containing $V(C)$.

\begin{prop}\label{F2}
(1) For a cycle $C$ of length at least $k$, $|H(C)|\ge k$ and every pair of vertices in $H(C)$ do not have common neighbours out of $H(C)$.

(2) If $C_{1}$  and $C_{2}$ are two cycles in $G$ with length at least $k$, then either $H(C_{1})=H(C_{2})$ or $H(C_{1})\cap H(C_{2})=\emptyset$. Moreover, if $H(C_{1})\cap H(C_{2})=\emptyset$ then $|E(H(C_1), H(C_2))|\le 1$.

(3) For a cycle $C$ of length at least $k$, $H(C)$ is well defined.

\end{prop}
\begin{proof}
(1) Clearly, $|H(C)|\ge |V(C)|\ge k$. If there are two vertices $x,y\in{H(C)}$ sharing  a common neighbour $z$ out of $H(C)$, then the path $xzy$ has both end vertices in $H(C)$. By (ii) of Proposition~\ref{F1}, $H(C)\cup \{x,y,z\}\in \mathcal{S}(G)$, a contradiction to the maximality of $H(C)$.

(2) If $H(C_{1})\cap H(C_{2})\not=\emptyset$, according to (i) of Proposition~\ref{F1}, $H(C_{1})\cup H(C_{2})\in\mathcal{S}(G)$. By the maximality of $H(C_{1})$ and $H(C_{2})$, we have $H(C_{1})=H(C_{1})\cup H(C_{2})=H(C_{2})$. If $H(C_{1})\cap H(C_{2})=\emptyset$ and there are two edges $x_{1}y_{1}, x
_{2}y_{2}\in E(H(C_1), H(C_2))$ with $x_{1}, x_{2}\in H(C_{1})$ and $y_{1}, y_{2}\in H(C_{2})$. By (1), $x_{1}\not=x_{2}$ and $y_{1}\not=y_{2}$.
Then there is a path $P$ in $G[H(C_{2})]$ connecting $y_{1}$ and $y_{2}$. So $x_{1}y_{1}Py_{2}x_{2}$ is a path with both end vertices in $H(C_{1})$. By (ii) of Proposition~\ref{F1}, $H(C_{1})\cup V(P)\subseteq \mathcal{S}(G)$, again a contradiction to the  maximality of $H(C_1)$.

(3) It follows directly from (2).

\end{proof}

The following is the key lemma in the proof of Theorem~\ref{THM: Main}.

\begin{lem}\label{LEM: C}
Every graph $G\in\mathcal{D}_{n,k}$ contains a cycle of length at least $\frac{k}{2}$ such that one of its vertices has degree at least $k-1$.
\end{lem}
\begin{proof}
Suppose to the contrary that there is a graph $G\in\mathcal{D}_{n,k}$ containing no cycle of qualified length and vertex degree.
We do operations on $G$ as follows: Let $G_0=G$. For $i\ge 1$, if $G_{i-1}$ contains a cycle $C_{i}$ of length at least $k$ then set $G_{i}=G_{i-1}-H(C_{i})$.
The procedure stops at step $m$ if $G_m$ contains no cycle of length at least $k$. By Theorem~\ref{THM: EG}, $G_0$ contains a cycle of length at least $k$. So $m\ge 1$.
By (2) of Proposition~\ref{F2},
$H(C_{1}), H(C_{2}),\ldots,H(C_{m})$ are pairwise disjoint and $G_{m}=G-H(C_{1})-H(C_{2})-\ldots-H(C_{m})$.
\begin{claim}\label{F3}
(I) Any vertex in $H(C_{i})$ has degree less than $k-1$.

(II) Any cycle of length at least $k$ in $G$ must be contained in some $G[H(C_{i})]$.
\end{claim}.
(I) For any vertex $u\in H(C_{i})$, there is a cycle in $G[H(C_{i})]$ containing $u$ of length at least $\frac{k}{2}$. If $d_{G}(u)\geq{k-1}$ then $C_i$ is a qualified cycle in $G$. It is a contradiction to $G$ is a counterexample.

 (II) Let $C$ be a cycle in $G$ of length at least $k$.  If $V(C)$ is disjoint with $H(C_{i})$ for any $i\in \{1,2,\ldots, m\}$, then $C$ must be a cycle of length at least $k$ in $G_{m}$, a contradiction the definition of $G_{m}$. So $V(C)\cap H(C_{i})\not=\emptyset$ for some $i\in\{1,2,\ldots, m\}$. If $V(C)\not\subset{H(C_{i})}$ then $H(C_{i})\cup{V(C)}\in\mathcal{S}(G)$, which is bigger than $H(C_i)$. This is a contradiction to the maximality of $H({C_i})$ containing $C_i$ in $\mathcal{S}(G)$.

Suppose that $V(G_{m})=\{y_{1},y_{2},\dots,y_{s}\}$. We construct a new graph $G'$ as below.
$$V(G')=\{x_{1},x_{2},\dots,x_{m},y_{1},y_{2},\dots,y_{s}\}, \mbox{ and } E(G')=E_{1}\cup E_{2}\cup E(G_{m}),$$
where
$$E_{1}=\{x_{i}x_{j} : i\not=j, e_{G}(H(C_{i}),H(C_{j}))>0\} \mbox{ and } E_{2}=\{x_{i}y_{j} : e_{G}(H(C_{i}),y_{j})>0\}.$$
By (1) and (2) of Proposition~\ref{F2}, $e_{G}(H(C_{i}), y_{j})\le 1$ and $e_{G}(H(C_{i}), H(C_{j}))\le 1$ for $i\not=j$.
So
\begin{eqnarray*}
|E_{1}|
&=&\sum_{1\le i<j\le m} e_{G}(H(C_{i}), H(C_{j}))\\ 
&=& e\left(G\left[\bigcup_{i=1}^{m}H(C_{i})\right]\right)-\sum_{i=1}^{m}e(G[H(C_{i})]),
\end{eqnarray*}
and
\begin{eqnarray*}
|E_{2}|
&=&\sum_{i=1}^{m}\sum_{j=1}^{s}e_{G}(H(C_{i}),y_{j})=\sum_{i=1}^{m}e_{G}(H(C_{i}),V(G_{m}))\\
&=&e_G\left(\bigcup_{i=1}^{m}H(C_{i}),V(G_{m})\right).
\end{eqnarray*}
By (I) of Claim~\ref{F3},
\begin{eqnarray*}
e(G[H(C_{i})])&\le &\frac 12\sum_{x\in H(C_{i})}d_{G}(x)\leq\frac{(k-2)|H(C_{i})|}{2}\\
&=&\frac{(k-1)(|H(C_{i})|-1)+(k-1)-|H(C_{i})|}{2}\\
&<&\frac{(k-1)(|H(C_{i})|-1)}{2} \,\, (\mbox{since $|H(C_i)|\ge k$}).
\end{eqnarray*}
Therefore,
\begin{eqnarray*}
e(G')&=&|E_{1}|+|E_{2}|+e(G_{m})\\
&=&e\left(G\left[\bigcup_{i=1}^{m}H(C_{i})\right]\right)+e_G\left(\bigcup_{i=1}^{m}H(C_{i}),V(G_{m})\right)+e(G_{m})-\sum_{i=1}^{m}e(G[H(C_{i})])\\
&=&e(G)-\sum_{i=1}^{m}e(G[H(C_{i})])\\
&>&\frac{n(k-1)}{2}-\sum_{i=1}^{m}\frac{(k-1)(|H(C_{i}|-1)}{2}\\
&=&\frac{k-1}{2}\left(\left(n-\sum_{i=1}^{m}|H(C_{i})|\right)+m\right)\\
&=&\frac{(k-1)(s+m)}{2}=\frac{(k-1)|V(G')|}{2}.
\end{eqnarray*}
We conclude that $G'\in\mathcal{D}_{|V(G')|, k}$. By Theorem~\ref{THM: EG}, $G'$ contains a cycle $C'$ of length at least $k$. Note that $H(C_1), H(C_2),\ldots, H(C_m)$ are pairwise disjoint. We can find a cycle $C$ in $G$ by replacing each $x_{i}$ contained in $C'$ by a path (possibly be a vertex) in $G[H(C_{i})]$. Then $C$ is a cycle of length at least $k$ in $G$. But $C$ is not contained in any $G[H(C_{i})]$, a contradiction to (II) of Claim~\ref{F3}.
\end{proof}

Now we are ready to give the proof of Theorem~\ref{THM: Main}.

\begin{proof}[\bf Proof of Theorem~\ref{THM: Main}]
Since $G\in\mathcal{D}_{n,k}$, $G$ contains a cycle $C$ of length at least $\frac k2$ such that there is a vertex $u\in V(C)$ of degree at least $k-1$. Assume that $|N_{G}(u)\cap V(C)|=t$ and $s=d_{G}(u)-t$. Then $s$ is the number of neighbours of $u$ out of $V(C)$. Now we label the neighbours of $u$ in $V(C)$ starting from $u$ along $C$ in clockwise, the first neighbour of $u$ is $u_{1}$, the second one is $u_{2}$, $\ldots$, the last one is $u_{t}$. The neighbours of $u$ out of $C$ are $v_{1}, v_{2},\ldots, v_{s}$.

If $r\leq s$ then we find a keyring $K$ from $C$ by appending vertices $v_{1}, v_{2},\ldots, v_{r}$  to $u$ as leaves. Clearly, $K$ contains $r$ leaves and  at least $\frac k2+r\geq k$ edges because $r\ge \frac k2$.
If $r>s$ then $r-s+1<t$ since $r\leq k-3$. Let $P$ be the segment on $C$ with end vertices $u_{r-s+1}$ and $u$ containing all $u_{i}$ with $r-t+1\leq i\leq t$. 
Appending the edge $uu_{r-s+1}$ to $P$ we get a cycle and then appending $r$ leaves $v_{1}, \ldots, v_{s}, u_{1}, u_{2},\ldots, u_{r-s}$ to the vertex $u$ of the cycle we obtain a keyring $K$ in $G$ with $r$ leaves.  Note that $K$ contains at least $k$ vertices and so $K$ contains at least $k$ edges.

\end{proof}

\section{Remarks}
Tur\'an asked the very problems for which the Erd\H{o}s-Gallai theorem provides the answers.
Erd\H{o}s-Gallai theorem solved the problem for cycles; Fan and Sun~\cite{Fan-Sun07} solved the lasso problem; Combining  Sidorenko's theorem and Theorem~\ref{THM: Main}, the keyring problem  is solved completely. Note that {A lasso (keyring) } is a graph obtained from a cycle by appending a path (a star) to one of its vertices. It is reasonable to consider the Tur\'an problem for a graph obtained from a cycle by appending a tree to one of its vertices (we may call it a {\em generalized keyring}). It will be plausible if the Erd\H{o}s-Gallai theorem provides the answer. We leave this as an open problem.

\end{document}